\documentclass[11pt,reqno,tbtags]{amsart}
\usepackage{graphicx}
\usepackage[colorlinks=true, pdfstartview=FitV, linkcolor=red, citecolor=blue, urlcolor=blue,pagebackref=false]{hyperref}
\usepackage{amssymb,amsmath, amsfonts}

\usepackage{epstopdf}
\newcommand{\E}{\mathbb E}

\newcommand{\1}{\mathbf{1}}
\newcommand{\x}{\mathbf{x}}
\newcommand{\y}{\mathbf{y}}
\newcommand{\prob}{\mathbb P}
\newcommand{\etX}{X^{\eta}}
\newcommand{\rW}{W^{\rho}}

\newcommand{\iv}{I^{\mathrm{Var}}}
\newcommand{\BKKKL}{\text{BKKKL}}
\newcommand{\f}{\check{f}}
\newcommand{\g}{\check{g}}

\newcommand{\gcal}{\mathcal G}
\newcommand{\Real}{\mathbb R}

\newcommand{\ig}{I^{\gcal}}
\newcommand{\eps}{\epsilon}

\newcommand{\sr}{s^\rho}
\newcommand{\cov}{\mathrm{Cov}}
\newcommand{\var}{\mathrm{Var}}
\newcommand{\ga}{\alpha}

\newcommand{\zg}{Z^{\mathcal {G}}}
\newcommand{\vg}{\mathrm{VAR}^{\mathcal {G}}}

\theoremstyle{definition}
\newtheorem{example}{Example}[section]
\newtheorem{definition}[example]{Definition}

\newtheorem{observation}[example]{Observation}

\theoremstyle{plain}

\newtheorem{lemma}[example]{Lemma}
\newtheorem{proposition}[example]{Proposition}
\newtheorem{theorem}[example]{Theorem}

\newtheorem{corollary}[example]{Corollary}

\theoremstyle{remark}

\newtheorem{remark}[example]{Remark}
\numberwithin{equation}{section}

\begin{document}

\title{Geometric Influences II: Correlation Inequalities and Noise Sensitivity}

\author{Nathan Keller$^*$}
\email{nathan.keller@math.biu.ac.il}

\thanks{$^*$ Bar Ilan University. Part of the work was done while the author was with
the Weizmann Institute of Science and was supported by the
Koshland Center for Basic Research.}

\address{Department of Mathematics, Bar Ilan University, Ramat Gan, Israel.}

\author{Elchanan Mossel $^\dagger$}
\email{mossel@stat.berkeley.edu}
\thanks{$^\dagger$ Weizmann Institute of Science and U.C. Berkeley. Supported by by NSF grant DMS 1106999 and DOD ONR grant N000141110140 and  by ISF grant 1300/08}
\address{Dept. of Statistics, 367 Evans Hall Berkeley, CA 94720.}

\author{Arnab Sen $^\ddagger$}
\email{a.sen@statslab.cam.ac.uk}

\address{Statistical Laboratory, Dept.\ of  Pure Mathematics and Mathematical Sciences, Wilberforce Road, CB3 0WB, UK.}
\thanks{$\ddagger$ Cambridge University.   Supported by EPSRC grant EP/G055068/1.}
\thanks{}




\begin{abstract}
In a recent paper, we presented a new definition of influences in
product spaces of continuous distributions, and showed that
analogues of the most fundamental results on discrete influences,
such as the KKL theorem, hold for the new definition in Gaussian
space. In this paper we prove Gaussian analogues of two of the
central applications of influences: Talagrand's lower bound on the
correlation of increasing subsets of the discrete cube, and the
Benjamini-Kalai-Schramm (BKS) noise sensitivity theorem. We then
use the Gaussian results to obtain analogues of Talagrand's bound
for all discrete probability spaces and to reestablish analogues
of the BKS theorem for biased two-point product spaces.
\end{abstract}

\maketitle

\section{Introduction}

\begin{definition}
Consider the discrete cube $\{-1,1\}^n$ endowed with the uniform
measure $\nu^{\otimes n} = (\frac12 \delta_{-1} + \frac12
\delta_{1})^{\otimes n}$, and let $f:\{-1,1\}^n \rightarrow
\mathbb R$. The influence of the $i$-th coordinate on $f$ is
defined as
\begin{equation}\label{def:inf}
I_i(f):= \mathbb{E_\nu} \Big[ \big|f(X)-
f(X^{[i]})\big| \Big],
\end{equation}
where $X = (X_1, \ldots, X_n)$ is a random vector in $\{-1,1\}^n$ distributed according to the measure $\nu^{\otimes n}$, and $X^{[i]}$ denotes the vector obtained from
$X$ by replacing  $X_i$ by~$-X_i$ and leaving the other
coordinates unchanged. The subscript $\nu$ in $\E_\nu$ emphasizes the fact that the expectation is taken w.r.t.\ the measure $\nu^{\otimes n}$. For a subset $A$ of the discrete cube
$\{-1,1\}^n$, we write $I_i(A)$ as a shorthand for $I_i(1_A)$, and
refer to it as the influence of the $i$-th coordinate on $A$. \end{definition}
The notion of influences of variables on Boolean functions is one
of the central concepts in the theory of discrete harmonic
analysis. In the last two decades it found several applications in
diverse fields, including Combinatorics, Theoretical Computer
Science, Statistical Physics, Social Choice Theory, etc.\ (see, for
example, the survey articles \cite{kalai06,Ryan-Survey}).

Two of the central applications are Talagrand's lower bound on the
correlation between increasing subsets of the discrete cube~\cite{Talagrand2}
and the Benjamini-Kalai-Schramm (BKS) theorem on noise sensitivity~\cite{BKS}.

Talagrand's result is an improvement over the classical
Harris-Kleitman correlation inequality~\cite{Harris,Kleitman}
stating that any two increasing (see
Definition~\ref{def:increasing} below) subsets of the discrete
cube are non-negatively correlated.
\begin{theorem}[Talagrand] \label{Thm:Talagrand-Correlation}
For any pair of increasing subsets $A,B \subset \{-1,1\}^n$,
\[
\nu^{\otimes n}(A \cap B)-\nu^{\otimes n}(A)\nu^{\otimes n}(B)
\geq c \varphi \left(\sum_{i=1}^n I_i(A) I_i(B) \right),
\]
where $\varphi(x)=x/\log(e/x)$, and $c>0$ is a  universal constant.
\end{theorem}

The BKS theorem deals with the sensitivity a of Boolean function
(or equivalently, a subset of the discrete cube) to a small random
perturbation of its input.
\begin{definition}
For a function $f : \{-1, 1\}^n \to \Real$, and for $\eta \in (0, 1)$, let
\[
Z(f, \eta)  = \E[ f(X)f(\etX)],
\]
where $X = (X_1,  \ldots, X_n)$ is uniformly distributed in $\{-1,
1\}^n$ and $\etX = (\etX_1, \ldots, \etX_n)$ is a $(1-
\eta)$-correlated copy of $X$. (This means that for $j \in \{1,2,
\ldots, n\}$, $\etX_j = X_j$ with probability $1-\eta$ and $\etX_j
= X_j'$ with probability $\eta$, independently for distinct $j$'s,
where $X' = (X_1', \ldots, X_n')$ is an i.i.d.\ copy of $X$).
Following Benjamini, Kalai and Schramm \cite{BKS}, we denote
\[
\mathrm{VAR}(f,\eta) = Z(f, \eta)  - \E[f(X)]^2.
\]
For a set $B \subseteq \{-1, 1\}^n$, and for $\eta \in (0, 1)$, we write
\[
Z(B, \eta)  = Z(1_B, \eta) \ \ \ \text{ and } \ \  \
\mathrm{VAR}(B,\eta) = \mathrm{VAR}(1_B,\eta).
\]

A sequence of sets $B_\ell \subseteq \{-1, 1\}^{n_\ell}$ is said
to be {\em asymptotically noise sensitive} if
\begin{equation}  \label{eq:asymp_ns_condition}
\lim_{ \ell \to \infty} \mathrm{VAR}(B_\ell,\eta)= 0 \quad \text{
for each } \eta \in (0, 1).
\end{equation}
\end{definition}

In a seminal paper, Benjamini, Kalai and Schramm~\cite{BKS} proved that a sequence of sets $B_\ell \subseteq \{-1, 1\}^{n_\ell}$ is asymptotically noise sensitive if the sum of the squares of the influences $\sum_{i=1}^{n_\ell} I_i(B_\ell)^2 $ goes to zero as $\ell \to \infty$.
Recently, Keller and Kindler~\cite{Keller-Kindler} obtained a quantitative
version of the BKS theorem.
\begin{theorem}[Quantitative BKS theorem]\label{Quantitative-BKS}
For any $n$, for any function $ f: \{-1,1\}^n \to [0,1]$, and for any $\eta
\in (0, 1)$,
\[
\mathrm{VAR}(f,\eta) \leq c_1 \cdot \left( \sum_{i=1}^n I_i(f)^2
\right)^{c_2 \cdot \eta},
\]
where $c_1,c_2$ are positive universal constants.
\end{theorem}

 The basic results on influences were obtained for
functions on the discrete cube, but some applications required
generalization of the results to more general product spaces.
Unlike the discrete case, where there exists a single natural
definition of influence, for general product spaces several
definitions were presented in different
papers, see for example~\cite{bourgain92,hatami09,keller11}.
In~\cite{Geom-Influences}, we presented a new notion
of influences in product spaces of continuous distributions, which we called geometric influences,
and proved analogues of the fundamental results on influences,
such as the Kahn-Kalai-Linial (KKL) theorem~\cite{KKL} and Talagrand's
influence sum bound~\cite{Talagrand1}, for geometric influences.

\medskip

In this paper we prove analogues of Talagrand's lower bound on
the correlation of increasing sets (Theorem~\ref{Thm:Talagrand-Correlation}
above) and of the quantitative BKS theorem (Theorem~\ref{Quantitative-BKS}
above), that hold for the standard Gaussian measure in $\mathbb{R}^n$
with respect to geometric influences.

\begin{definition}
Let $\mu(dx) = (1/\sqrt{2 \pi}) \exp (  - x^2/2) dx$ be the
standard Gaussian measure on $\mathbb R$. Let $\phi$ (resp.\ $\Phi$) be the density (resp.\ distribution
function) of the Gaussian measure $\mu$ on $\Real$, and denote
$\bar \Phi(x) = 1 - \Phi(x)$. Given a Borel-measurable
set $A \subseteq \mathbb{R}$, its lower Minkowski content
$\mu^{+}(A)$ is defined as
\[
\mu^{+}(A) :=  \liminf_{ r \downarrow 0} \frac{\mu( A + [-r,
r]) - \mu(A)}{r}.
\]
For any Borel-measurable set $A \subseteq \mathbb{R}^n$, for each $ 1 \leq
i \leq n$ and an element $x =(x_1, x_2, \ldots, x_n) \in
\mathbb{R}^n$, the restriction of $A$ along the fiber of $x$ in
the $i$-th direction is given by
\[
A^{x}_i := \{ y \in \mathbb{R}: (x_1, \ldots,  x_{i-1}, y,
x_{i+1}, \ldots, x_n) \in A \}.
\]
The {\em geometric influence} of the $i$-th coordinate on $A$ is
\[
\ig_i(A) := \mathbb{E}_{x} [\mu^{+} (A^{x}_i)],
\]
that is, the expectation of $\mu^{+} (A^{x}_i)$ when $x$ is
chosen according to the measure $\mu$.
\end{definition}
We note that the geometric meaning of the influence is that for a
monotone (either increasing or decreasing) set $A$, the sum of
influences of $A$ is equal to the size of its boundary with
respect to a uniform enlargement (see \cite{Geom-Influences}).

\medskip

In the sequel, whenever we talk about sets or  functions in $\mathbb R^n$, we implicitly assume that they are Borel measurable. Our first result is a lower bound on the correlation between two increasing bounded functions in the Gaussian space.
\begin{theorem} \label{thm:gaussian_talagrand}
 Let $\varphi(x) = x /\log (e/x)$. There exists a universal constant $c>0$ such that for
 any $n \in \mathbb{N}$ and for any  two increasing subsets $A$ and $B$ of $\mathbb R^n$, we have
\[
\mu^{\otimes n}(A \cap B)  - \mu^{\otimes n}(A) \mu^{\otimes n} (B) \ge
c\varphi \big(\sum_{ i=1}^n I_i^{\mathcal G}(A)  I_i^{\mathcal G}(B)\big).
\]
\end{theorem}
We show that the assertion of the theorem is tight, up to the constant factor.
The proof of Theorem~\ref{thm:gaussian_talagrand} uses
Talagrand's  result for the discrete cube, along with appropriate
limit arguments.  By appealing to direct Gaussian arguments, we
obtain another lower bound on the correlation between  a pair of
increasing subsets in the Gaussian space.
\begin{theorem} \label{thm:alt_bound}
There exists a universal constant $c>0$ such that for
any $n \in \mathbb{N}$ and for any  two increasing subsets $A$ and $B$ of $\mathbb R^n$, we have
\[
\mu^{\otimes n}(A \cap B)  - \mu^{\otimes n}(A)\mu^{\otimes n}(B) \ge
c \sum_{i=1}^n \frac{ \ig_i(A)\ig_i(B)} {\sqrt{\log (e/\ig_i(A)) \log (e/\ig_i(B) )}}.
\]
\end{theorem}
In fact, we prove  functional versions of the above two theorems (see Theorem~\ref{thm:gaussian_talagrand_func} and Theorem~\ref{thm:alt_bound_fn}), which, with a little bit of extra work,  can then be applied to deduce the results for the characteristic functions of increasing sets.

Theorem~\ref{thm:alt_bound}  is neither uniformly stronger nor
uniformly weaker  than
Theorem~\ref{thm:gaussian_talagrand}, as there are cases
where each one beats the other. It should be noted that while
Talagrand's lower bound uses the classical Bonami-Beckner
hypercontractive inequality~\cite{Bonami,Beckner}, the proof of
Theorem~\ref{thm:alt_bound} uses Borell's reverse
hypercontractive inequality~\cite{Borell}. It will be interesting
to find out whether hypercontractivity and reverse
hypercontractivity can be combined to obtain a new lower bound
that will enjoy the benefits of both
Theorems~\ref{thm:gaussian_talagrand}
and~\ref{thm:alt_bound}.

Recall that the classical Gaussian FKG inequality~\cite{FKG} asserts that for any pair
 of coordinate-wise increasing functions $f, g : \mathbb R^n \to \mathbb R$, we have
\[ \E_\mu[ fg] \ge  \E_\mu[ f]  \E_\mu[ g] \]
Hence, Theorems~\ref{thm:gaussian_talagrand} and  \ref{thm:alt_bound} (or more appropriately their functional versions) provide quantitive
versions of the Gaussian FKG inequality.


\medskip

Our second result is a Gaussian analogue of the noise sensitivity
results of Benjamini-Kalai-Schramm \cite{BKS}.


\begin{definition}\label{def:GNS}
Let $W, W'$ be i.i.d.\ standard Gaussian vectors on $\Real^n$ and
let $\rW = \sqrt{1 - \rho^2} W + \rho W'$. For a function $f:
\mathbb{R}^n \to \mathbb R$, and for $\rho \in (0, 1)$, let
\[
\zg(f, \rho)  = \E[ f(W)f(\rW)],
\]
provided $\E[|f(W)|^2] < \infty$. Denote
\[
\vg(f,\rho) = \zg(f, \rho)  - \E [f(W)]^2.
\]
For a set $A \subset \mathbb{R}^n$, and for $\rho \in (0, 1)$, we
write
\[
\zg(A, \rho)  =  Z(1_A, \rho), \ \ \ \text{ and } \ \ \  \vg(A,\rho)  = \vg(1_A,\rho).\]
A sequence of sets $A_\ell \subseteq \Real^{n_\ell}$ is said to be
asymptotically Gaussian noise-sensitive if
\begin{equation}  \label{eq:asymp_Gaussian_ns_condition}
\lim_{ \ell \to \infty} \vg(A_\ell,\rho)= 0 \quad \text{ for each
} \rho \in (0, 1).
\end{equation}
\end{definition}
\begin{theorem}\label{thm:Gaussian-Quantitative-BKS}
For any $n \ge 1$, for any set $A \subset \Real^n$, and for any $\rho \in (0, 1)$,
\[
\vg(A,\rho) \leq C_1 \cdot \left( \sum_{i=1}^n \ig_i(A)^2
\right)^{C_2 \rho^2},
\]
where $C_1,C_2$ are positive universal constants.
\end{theorem}
The proof of Theorem~\ref{thm:Gaussian-Quantitative-BKS}  again relies upon an appropriate limit argument and uses Theorem~\ref{Quantitative-BKS} as a blackbox.


Theorems~\ref{thm:gaussian_talagrand} and
\ref{thm:alt_bound} allow us to obtain analogues of
Talagrand's lower bounds  for any discrete product probability space (see Theorem~\ref{Thm:Discrete1}), where the lower bound involves a discrete variant of the geometric influence, called  $h$-influence.
Theorem~\ref{thm:Gaussian-Quantitative-BKS} can be used to
obtain an analogue of the BKS theorem in the case of the discrete
hypercube $\{0,1\}^n$ endowed with a biased product measure (see Theorem~\ref{Thm:Discrete2}). We note that for the
biased product  measures on the discrete hypercube, these results were
previously obtained in~\cite{Keller-Reduction,Keller-Kindler} by
different methods. Comparison of our results with the results
of~\cite{Keller-Reduction,Keller-Kindler} suggests that, in some
sense, the $h$-influence obtained from the geometric influence  is more
natural than the notion of influences used for the biased measure
in previous works.

\medskip

This paper is organized as follows. In Section~\ref{sec:Talagrand}
we prove functional versions of Theorem~\ref{thm:gaussian_talagrand} and Theorem~\ref{thm:Gaussian-Quantitative-BKS}.
 In Section~\ref{sec:Direct}, we present a functional version of Theorem~\ref{thm:alt_bound}  using the Ornstein-Uhlenbeck semigroup theory.  In Section~\ref{smooth_approx} we give an argument to suitably approximate the characteristic functions of monotone sets by smooth functions  and apply it to deduce Theorems~\ref{thm:gaussian_talagrand}, \ref{thm:alt_bound}  and~\ref{thm:Gaussian-Quantitative-BKS} from their functional counterparts. We also discuss how Theorem~\ref{thm:gaussian_talagrand} and Theorem~\ref{thm:alt_bound} compare against each other.
Finally, we deduce the analogous statements for discrete product
probability spaces in Section~\ref{sec:Discrete}, and conclude the
paper with a few open problems in Section~\ref{sec:Open}.

\section{Refined Gaussian FKG Inequality and Gaussian BKS Theorem}
\label{sec:Talagrand}

The main goal of this section is to  we prove the following two theorems which are functional forms of Theorem~\ref{thm:gaussian_talagrand} and Theorem~\ref{thm:Gaussian-Quantitative-BKS}. Note that the role of Gaussian influences is now played by the $L^1$ norm of the partial derivatives of the functions.

 \begin{theorem} \label{thm:gaussian_talagrand_func}
 Let $\varphi(x) = x /\log (e/x)$. There exists a universal constant $c>0$ such that for
 any $n \ge 1$ and for any  two increasing continuously differentiable functions $f, g: \Real^n \to [-1, 1]$,  we have
\[
\E_\mu[ fg] - \E_\mu[ f]  \E_\mu[ g] \ge
c\varphi \Big(\sum_{ i=1}^n \E_\mu[\partial_i f]  \E_\mu[\partial_i g] \Big),
\]
where $\E_\mu$ stands for integration w.r.t.\ $\mu^{\otimes n}$.
\end{theorem}

\begin{theorem} \label{thm:Gaussian-Q-BKS_function}
For any $n \ge 1$, for any continuously  differentiable function $f :\Real^n \to [-1,1]$, and for any $\rho \in (0, 1)$,
\[
\vg(f,\rho) \leq C_1 \cdot \left( \sum_{i=1}^n \E_\mu[|\partial_i f|]^2
\right)^{C_2\rho^2},
\]
where $C_1,C_2$ are positive universal constants.
\end{theorem}

The
proof strategy is to approximate the functions  in the ``Gaussian world'' by
sequences of functions defined on the discrete cubes $\{-1,1\}^{n_{\ell}}$
(where $n_\ell \to \infty$), and to deduce the assertions of the
theorems by an appropriate limit argument from the corresponding theorems in the ``discrete world''.

For a function $f:\Real^n \to \Real$, we construct a sequence
$\{\f_m\}_{m=1}^{\infty}$ of functions as follows. For each $m \in
\mathbb{N}$, we denote elements in $\{-1,1\}^{mn}$ by vectors
$(x_1,x_2,\ldots,x_n)$, where each $x_i= (x_{i1}, x_{i2},  \ldots,
x_{im})$ is a vector in $\{-1,1\}^m$. We write $s_i=s_i(m)$ as a
shorthand for $m^{-1/2}\sum_{j=1}^m x_{ij}$ and let $s = (s_1,
\ldots, s_n) \in \Real^n$. Then, we define the function $\f_m :
\{-1,1\}^{mn} \to \Real$ by $\f_m(x_1,\ldots,x_n)= f(s_1, \ldots,
s_n)$. In order to simplify the notation, we leave the dependence
of $s$ on $m$ implicit in some of the places, and alert the reader
that in the sequel, $s$ always depends on $m$.  The next lemma is
our main tool for transferring the results from the discrete world
to the Gaussian world.

\begin{lemma} \label{lem:infl_approx_fn}
Fix $n \ge 1$ and $ 1 \le i \le n$. Let $f$  and $g$ be two continuously differentiable functions on $\mathbb R^n$ such that the partial derivatives $\partial_i f $ and $\partial_i g$ are  bounded. Then
\[
\sum_{j=1}^m  I_{ij} ( \f_m) I_{i j} ( \g_m) \to 4 \E_\mu[ | \partial_i f| ] \E_\mu [| \partial_i g|], \quad \text{ as  } m \to \infty.
\]
\end{lemma}

\begin{proof}
Since the functions $\f_m$ and $\check g_m$ are invariant under
permutations of the coordinates $\{x_{ij} \}_{ 1 \le j \le m}$ for
each fixed $i$,  it follows that  $\sum_{j=1}^m  I_{i j} ( \f_m)
I_{ij} ( \g_m) = m  I_{i1} ( \f_m) I_{i1} ( \g_m)$.  Thus, it
suffices to show that $ \sqrt{m} I_{i1} ( \f_m) \to  2\E_\mu[ |
\partial_i f| ]$ and similarly for $g$. Without loss of
generality,  we take $i=1$. We have
\[ I_{11} ( \f_m) =  \E_\nu \Big[  \big| f(s'_1 + m^{-1/2}, s_2, \ldots, s_n)] - f(s'_1- m^{-1/2}, s_2, \ldots, s_n) \big| \Big], \]
where $s'_1 = m^{-1/2} \sum_{ j=2}^m x_{1j}$. By the Mean Value
Theorem,
\[ \frac{ f(s'_1 + m^{-1/2}, s_2, \ldots, s_n)  -   f(s'_1 - m^{-1/2}, s_2, \ldots, s_n)}{ 2m^{-1/2} } = \partial_1 f(s_1' + \eps_m, s_2, \ldots, s_m), \]
where $\eps_m$ is an error term  that depends on $s'_1, s_2,
\ldots, s_n$, and whose absolute value is bounded by $m^{-1/2}$.
Therefore, we obtain
\[ \sqrt{m} I_{11} ( \f_m) = 2 \E_\nu \Big[  \big| \partial_1 f(s_1' + \eps_m, s_2, \ldots, s_m) \big| \Big]. \]
Since $(s_1' + \eps_m, s_2, \ldots, s_m)$  converges in
distribution to $\mu^{\otimes n}$, and since $\partial_1 f$ is a
continuous, bounded function, we conclude that
\[ \lim_{ m \to \infty} \sqrt{m} I_{11} ( \f_m)  = 2\E_\mu[ | \partial_1 f|].\]
The assertion of the lemma follows.
\end{proof}

To prove Theorem~\ref{thm:gaussian_talagrand_func}, we will need
the following  functional version of Talagrand's inequality on the
discrete cube.
\begin{theorem} \label{Thm:Talagrand-Correlation_func}
For any $n \ge 1$ and for any pair of increasing functions $f,g :\{-1,1\}^n \to [0, 1]$,
\[
\E_\nu[ fg] - \E_\nu[ f]  \E_\nu[ g]
\geq c \varphi \left(\sum_{i=1}^n I_i(f) I_i(g) \right),
\]
where $\varphi(x)=x/\log(e/x)$, and $c>0$ is a  universal constant.
\end{theorem}
This version is obtained by following Talagrand's proof
step-by-step, using the fact that for a monotone function $f$,
$I_i(f)$ is equal in absolute value to the coefficient $\hat
f(\{i\})$ in the standard Fourier-Walsh expansion of $f$. The
exact proof (of a slightly more general statement) appears
in~\cite{keller09}.


\begin{proof}[Proof of Theorem~\ref{thm:gaussian_talagrand_func}]
Note that since $f, g$ are increasing and bounded, by the
Fundamental Theorem of Calculus, $\partial_i f,  \partial_i g$ are
nonnegative and integrable. In particular,  we have $ 0 \le
\E_\mu[\partial_i f],  \E_\mu[\partial_i g]  < \infty$ for all
$i$. First we assume that  $f,g$ are increasing $C^1$ functions on
$\mathbb{R}^n$ such that both $f, g$ take values in $[0,1]$ and
$\|\partial_i f\|_\infty, \|\partial_i g\|_\infty < \infty$ for
all $i$. It follows from
Theorem~\ref{Thm:Talagrand-Correlation_func} that there exists a
universal constant $c>0$ such that for each $m \in \mathbb{N}$, we
have
\begin{equation}\label{ineq:dis_tala}
\int \f_m \g_m d\nu^{\otimes nm}  -\int \f_m  d\nu^{\otimes nm}  \int  \g_m d\nu^{\otimes nm}  \ge c\varphi
\Big(\sum_{ i=1}^n \sum_{j=1}^m  I_{i j}(\f_m)
I_{ij}(\f_m)\Big).
\end{equation}
By the Central Limit Theorem, $s(m) = (s_1, \ldots, s_n)$
converges in distribution to $\mu^{\otimes n}$ as $m \to \infty$.
Thus,  the left hand side of \eqref{ineq:dis_tala} converges to
$\E_\mu[ fg] - \E_\mu[ f]  \E_\mu[ g] $ as $m \to\infty$. On the
other hand, by letting $m \to \infty$ and applying
Lemma~\ref{lem:infl_approx_fn} to the right hand side of
\eqref{ineq:dis_tala}, we obtain
\begin{equation}\label{ineq:fg}
\E_\mu[ fg] - \E_\mu[ f]  \E_\mu[ g] \ge
c\varphi \Big(\sum_{ i=1}^n \E_\mu[\partial_i f]  \E_\mu[\partial_i g] \Big).
\end{equation}
We can easily extend the above inequality, with the constant $c$
replaced by a new constant $c/ (1+\log 2)$,  to increasing $C^1$
functions  $f,g$  such that both $f, g$ take values in $[-1,1]$
and $\|\partial_i f\|_\infty, \|\partial_i g\|_\infty < \infty$
for all $i$. To do that we apply \eqref{ineq:fg} for  the
functions $(1+f)/2,  (1+g)/2$  and note that $2 \varphi(x/2) \ge
\frac{1}{1+\log 2} \varphi(x)$ for all $x \in [0, 1]$.

Now we want to remove the  condition that the partial derivatives
of $f, g$ are bounded.  Let $f, g$ be as given in the hypothesis
of Theorem~\ref{thm:gaussian_talagrand_func}. For $K>0$, set $J_K
= [-K, K]^n, M_K = f(K, \ldots, K)$, and $m_K = f(-K, \ldots,
-K)$. Since $f$ is increasing, $M_K = \max_{ x \in J_K} f(x)$ and
$m_K = \min_{ x \in J_K} f(x)$. Let $f_K = \min( \max(f, m_K),
M_K)$. Hence, $f_K \equiv f$ inside $J_K$. Let $\eta \in
C^\infty(\Real^n)$ be the standard mollifier, that is, $\eta(x) =
C \exp \left( \frac{1}{|x|^2 -1}\right)\1_{|x| \le 1}$, where the
constant $C>0$ is selected so that $\int_{\Real^n} \eta(x) dx =
1$. For  each $ \eps >0$, set  $\eta_\eps(x) : = \eps^{-n} \eta(x/
\eps)$. Finally, define $f_{K, \eps} = f_K * \eta_\eps =
\int_{\Real^n} f_K(x-y)\eta_\eps(y)dy$. From the standard
properties of the mollifier, it follows that $ f_{K, \eps} \in
C^\infty(\Real^n)$, $f_{K, \eps}$  is increasing and  $|f_{K,
\eps}| \le 1$.  Note that for any $h \in \Real$, for any $z \in
\Real^n$,
\[ 0\le \frac{f_K(z + e_i h)  - f_K(z)}{h} \le \frac{f(z + e_i h)  - f(z)}{h}, \]
$e_i$ being the $i^{th}$ coordinate vector in $\Real^n$. It
follows that $0 \le \partial_i f_{K, \eps}   \le  \partial_i (f *
\eta_\eps) =  \partial_i f * \eta_\eps $.

Given $\delta>0$, we claim that there exist $K>0$ and $\eps>0$ such that $\int_{\mathbb R^n} | f_{K, \eps}  - f|^2 d \mu^{\otimes n} < \delta$ and  $\int_{\mathbb R^n} | \partial_i f_{K, \eps}  - \partial_i f| d \mu^{\otimes n} < \delta$.  To prove the claim, first find $K>0$ large such that $ |\int_{J_{K-1/2}^c} \partial_i f d \mu^{\otimes n}| < \delta/3$. For $0< \eps< 1/2$,  $\partial_i f_K = 0$ outside $J_{K+1/2}$ and
 we estimate \begin{align*}
\int_{\mathbb R^n} | \partial_i f_{K, \eps}  - \partial_i f| d \mu^{\otimes n}  &\le \int_{J_{K+1/2}} | \partial_i f_{K, \eps}  - \partial_i f *\eta_\eps| d \mu^{\otimes n} + \int_{J_{K+1/2}} |\partial_i f *\eta_\eps -   \partial_i f | d \mu^{\otimes n}   \\
& + \int_{J_{K+1/2}^c} \partial_i f d \mu^{\otimes n}.
\end{align*}
Note that whenever $\eps \in (0, 1/2)$, $\partial_i f_{K, \eps} =
\partial_i f * \eta_\eps$   on $J_{K-1/2}$. Hence,
\[\int_{J_{K+1/2}} | \partial_i f_{K, \eps}  - \partial_i f *\eta_\eps| d \mu^{\otimes n} \le \int_{J_{K+1/2} \cap J_{K-1/2}^c} |\partial_i f *\eta_\eps| d \mu^{\otimes n}.\]
By the well-known property of the mollifier, $ \partial_i f *\eta_\eps \stackrel{L^p}{\to} \partial_i f $ for any $ 1 \le p < \infty$ over compact sets. Thus, by choosing $\eps>0$ small we can make
$\int_{J_{K+1/2}} |\partial_i f *\eta_\eps -   \partial_i f | d \mu^{\otimes n} < \delta/3  $ and $ \int_{J_{K+1/2} \cap J_{K-1/2}^c} |\partial_i f *\eta_\eps| d \mu^{\otimes n} \le  \int_{J_{K+1/2} \cap J_{K-1/2}^c} |\partial_i f | d \mu^{\otimes n} + \delta/3$ and hence, $\int_{\mathbb R^n} | \partial_i f_{K, \eps}  - \partial_i f| d \mu^{\otimes n} < \delta$.

On the other hand, note that
\begin{align*}
\int_{\mathbb R^n} | f_{K, \eps}  - f|^2 d \mu^{\otimes n}  &\le 2 \int_{\mathbb R^n} | f_{K, \eps}  - f * \eta_\eps|^2 d \mu^{\otimes n} + 2 \int_{\mathbb R^n} | f * \eta_\eps - f|^2 d \mu^{\otimes n}.
\end{align*}
For $\eps>0$ fixed, $ \int_{\mathbb R^n} | f_{K, \eps}  - f * \eta_\eps|^2 d \mu^{\otimes n} \to 0$ as $K \to \infty$ by dominated convergence.  Since $ f * \eta_\eps \to f$ pointwise as $\eps \to 0$, the second integral also goes to zero  by dominated convergence. Thus we establish our claim.

Now note that \eqref{ineq:fg} holds for functions $f_{K, \eps}$ and $g_{K, \eps}$. We complete the proof of the theorem by  approximating the original functions $f$ and $g$ by $f_{K, \eps}$ and $g_{K, \eps}$ with suitably  large $K$ and small $\eps$.
\end{proof}


Now we prove
Theorem~\ref{thm:Gaussian-Q-BKS_function}, thus obtaining a
Gaussian analogue of the quantitative BKS
theorem~\cite{Keller-Kindler}.

\begin{proof}[Proof of Theorem~\ref{thm:Gaussian-Q-BKS_function}]
Assume first that $f$ is continuously differentiable with bounded
partial derivatives. We apply Theorem~\ref{Quantitative-BKS} to
the approximating  functions $\f_m:\{-1,1\}^{nm} \to [0,1]$ to
obtain, for any $\eta \in (0,1)$ and  any $m \ge 1$,
\begin{equation}\label{ineq:approx_gns}
\mathrm{VAR}(\f_m,\eta) \leq c_1 \cdot \left( \sum_{i=1}^n \sum_{j=1}^m I_{ij}(\f_m)^2
\right)^{c_2 \cdot \eta},
\end{equation}
where $c_1>0,c_2>0$ are universal constants. We claim that
$\mathrm{VAR}(\f_m,1 - \sqrt{1- \rho^2}) \to \vg(f, \rho)$ as $m
\to \infty$. Let $\x = (x_{ij})_{1 \le i \le n, 1 \le j \le m}$
and $\y = (y_{ij})_{1 \le i \le n, 1 \le j \le m}$ be $\sqrt{1 -
\rho^2}$ correlated vectors that are uniformly distributed in $\{
-1, 1\}^{mn}$.  Set $s(m) = (s_1, \ldots, s_n)$ and  $\sr(m) =
(\sr_1, \ldots, \sr_n)$ where $s_i = m^{-1/2}\sum_{j=1}^m x_{ij}$
and $\sr_i = m^{-1/2}\sum_{j=1}^m y_{ij}$. By definition, $Z(\f_m,
1 - \sqrt{1- \rho^2}) = \E_\nu[ f(s(m))f(\sr(m))]$. By the Central
Limit Theorem, $(s(m), \sr(m))$ converges in distribution to
$(W,\rW) $ as $m \to \infty$. Since the map $(z, z') \mapsto
f(z)f(z')$ is bounded and continuous on $\Real^{2n}$, it follows
that $\lim_{m \to \infty} Z(\f_m, 1 - \sqrt{1- \rho^2}) = \zg(f,
\rho)$.  That $\E_\nu [\f_m(\x)] =\E_\nu[ f(s(m))] $ converges to
$\E_\mu [f(W)]$ as $m \to \infty$  is again an immediate
consequence of the Central Limit Theorem. This proves the claim.
By letting $m \to \infty$ in \eqref{ineq:approx_gns} with $\eta =
1 - \sqrt{1 - \rho^2} \in (0,1)$  and by virtue of  the above
claim and Lemma~\ref{lem:infl_approx_fn}, we obtain the following
inequality for the function $f$ with $C_2 = c_2$ and $C_1' =
4c_1$,
\begin{equation} \label{ineq:discrete_to_Gaussian_BKS}
\vg(f,\rho) \leq C_1' \cdot \left( \sum_{i=1}^n \E_\mu[|\partial_i f|]^2
\right)^{C_2'(1 - \sqrt{1 - \rho^2})}.
\end{equation}
Extending \eqref{ineq:discrete_to_Gaussian_BKS} to $C^1$ functions
$f$ with bounded partial derivatives which  take values in
$[-1,1]$ instead of $[0,1]$ is fairly straightforward and can be
achieved  (with $C_1 = 2C_1'$) by arguing with the function
$(1+f)/2$ which now takes values in $[0,1]$. If $\sum_{i=1}^n
\E_\mu[|\partial_i f|]^2 \le 1$, then by observing the simple fact
that  $1- \sqrt{1 - \rho^2} \ge \rho^2/2$ for all $\rho \in (0,
1)$,  we get the desired inequality (with $C_2 = C_2'/2$) for the
function $f$.  On the other hand,  if  $\sum_{i=1}^n
\E_\mu[|\partial_i f|]^2 >1$, then the assertion of the theorem
trivially holds for $f$ since $\vg(f,\rho) \le |f| \le 1$.

Now take a general $C^1$ function $f: \mathbb R^n \to [-1,1]$. If
$\E_\mu[|\partial_i f|] = \infty$ for some $i$, then the theorem
holds trivially. So, assume that $\E_\mu[|\partial_i f|] < \infty$
for all $i$. Let $\eta, \eta_\eps $ and $J_K$ be as above.  Define
$f_K = f \1_{J_K}$ and $f_{K, \eps} = f_K * \eta_\eps$. Clearly,
$f_{K, \eps}$ is $C^\infty$ and $|f_{K, \eps}| \le 1$ and
$|\partial_i f_{k, \eps}|$ are bounded  for all $i$ (since $f_{K,
\eps}$ is compactly supported).  Note that as $K \to \infty, \eps
\to 0$, $f_{K, \eps}(z) \to f(z)$ pointwise, and hence by
dominated convergence, $\vg(f_{K, \eps}, \rho) \to \vg(f, \rho)$.
Next we prove that $\E_\mu[|\partial_i f_{K, \eps}  - \partial_i
f|] \to 0$ as $K \to \infty, \eps \to 0$. Towards this end, we
bound
\begin{align}\label{bound_decomp}
\E_\mu[|\partial_i f_{K, \eps}  - \partial_i f|]  &\le \E_\mu[|\partial_i f_{K, \eps}  - \partial_i f| \1_{ J_{K - \eps}}] + \E_\mu[|\partial_i f_{K, \eps}  - \partial_i f| \1_{ J_{K+\eps} \cap J_{K - \eps}^c}] \notag \\ &+ \E_\mu[|\partial_i f_{K, \eps}  - \partial_i f| \1_{ J_{K+\eps}^c}] \notag \\
 &\le \E_\mu[|\partial_i (f * \eta_\eps)  - \partial_i f| \1_{ J_{K}}] + \E_\mu[|\partial_i f_{K, \eps}|   \1_{ J_{K+\eps} \cap J_{K - \eps}^c}]  + \E_\mu[|\partial_i f| \1_{ J_{K}^c}].
\end{align}
Note that $ \partial_i f_{K, \eps}(x)  = \int f_K(y) \partial_i
\eta_\eps(x-y)dy = \eps^{-1} \int \partial_i \eta(z) f_K(x- \eps
z)dz$. Since  $|f_K|$ is  bounded by $1$, and $\int |\partial_i
\eta(z)| dz < \infty$, we have  $|\partial_i f_{K, \eps} | \le
C'\eps.$ Thus  the second expectation in \eqref{bound_decomp} can
be bounded above by $ C'\eps^{-1} \mu^{\otimes n} (J_{K+\eps} \cap
J_{K - \eps}^c) ) \le C'' \phi(K)$, where the constant $C''$ does
not depend on $K$ or $\eps$. The third  expectation in
\eqref{bound_decomp} can be made arbitrarily small by taking $K$
sufficiently large and the first expectation can be made as small
as we want choosing $\eps >0$ sufficiently small.  Therefore,
$\E_\mu[|\partial_i f_{K, \eps}|] - \E_\mu[ |\partial_i f|] \to 0$
as $K \to \infty, \eps \to 0$.

Clearly, the statement of the theorem holds for each $f_{K, \eps}$. Taking $K \to \infty, \eps \to 0$, we obtain the desirable conclusion for the original function $f$.
\end{proof}

\section{A Direct Gaussian Approach via the Ornstein-Uhlenbeck Semigroup}
\label{sec:Direct}

In this section we prove a functional version of Theorem~\ref{thm:alt_bound} (Theorem~\ref{thm:alt_bound_fn} below) and an
inverse Gaussian BKS theorem  using tools from the ``Gaussian
world'' without appealing to the corresponding results for the
discrete cube as we did in  the previous section.

 \begin{theorem} \label{thm:alt_bound_fn}
Let $f, g: \Real^n \to [-1,1]$ be increasing continuously
differentiable functions. Then
\[ \E_\mu [fg] - \E_\mu [f] \E_\mu [g] \ge  c \sum_{i=1}^n \frac{ \E_\mu [\partial_i f] \E_\mu[\partial_i g]} {\sqrt{\log (e/\E_\mu[\partial_i f]) \log (e/\E_\mu[\partial_i g] )}},\]
where $c>0$ is a universal constant.
\end{theorem}

We start with a few standard definitions and simple lemmas related
to the Ornstein-Uhlenbeck semigroup. For a more detailed treatment
of these notions, the reader is referred to~\cite{CEL,Ledoux2}.

\begin{definition}
Let  $(P_t)_{t\ge0}$ be the Ornstein-Uhlenbeck semigroup
associated with the generator $L = \Delta  - x \cdot \nabla$ on
$\Real^n$. This semigroup acts on the functions on $\Real^n$ as
follows:
\[
P_t f(x) = \int f( e^{-t} x + \sqrt{ 1- e^{-2t}}y)
\mu^{\otimes n}(dy), \quad x \in \Real^n.
\]
\end{definition}

It is well known that $(P_t)_{ t \ge 0}$ is reversible with the
invariant measure $\mu^{\otimes n}$. For $t>0$, the operator $P_t$
maps bounded measurable functions to $C^\infty$ functions. It also
maps an increasing function to an increasing function. The
following simple properties of the operator $P_t$ will be very
useful for later purposes:
\begin{observation}
Let $g:\Real^n \to \Real$ be a smooth function. Then:
\begin{enumerate}
\item[i.]
\begin{equation} \label{OU_fact:1}
 \partial_i P_t g  = e^{-t} P_t \partial_i g \quad \forall t \ge 0.
 \end{equation}

\item[ii.] If $|g(x)| \le 1$ for all $x$, then
 \begin{equation} \label{OU_fact:2}
 |\nabla P_t g | \le \frac{1}{\sqrt t} \quad 0< t\le 1/2.
 \end{equation}

\item[iii.] If $g$ is increasing, then
 \begin{equation} \label{OU_fact:3}
 \partial_i P_t g \ge 0  \quad \forall t \ge 0.
  \end{equation}
\end{enumerate}
\end{observation}

\begin{lemma} \label{eq:p_norm_bound}
Let $g$ be a smooth function with $|g(x)| \le 1$ for all $x$, and
let $t \in (0, 1/2]$.
\begin{enumerate}
\item[(i)] For $p \ge 1$, we have $\|\partial_i P_t g \|_p \le
t^{-(p-1)/2p}e^{-t/p}  \| \partial_i  g \|_1^{1/p}.$

\item[(ii)] Assume further that $g$ is increasing. Then for $0< p
<1$, we have $\|\partial_i P_t g \|_p \ge t^{(1-p)/2p}e^{-t/p} \|
\partial_i g \|_1^{1/p}.$
 \end{enumerate}
  \end{lemma}
  \begin{proof}
 \noindent (i)  By \eqref{OU_fact:1} and \eqref{OU_fact:2}, we have
 \[ \|\partial_i P_t g \|_p^p \le   t^{-(p-1)/2}\|\partial_i P_t g \|_1
 = t^{-(p-1)/2}e^{-t}  \| P_t(\partial_i  g) \|_1 \le
t^{-(p-1)/2}e^{-t}  \| \partial_i  g \|_1,\]
 where in the last inequality we use the fact that $P_t:L^{1}(\mu^{\otimes n}) \to L^{1}(\mu^{\otimes n})$ is a contraction.

  \noindent (ii) Again using \eqref{OU_fact:1} and \eqref{OU_fact:2}, we obtain
  \[ \| \partial_i P_t g \|_1 \le t^{ - (1-p)/2}  \|\partial_i P_t g \|_p^p. \]
Note that since $g$ is increasing,
\[ \| \partial_i P_t g \|_1 = \E_\mu[ \partial_i P_t g ] = e^{-t} \E_\mu[  P_t \partial_i g ] = e^{-t} \E_\mu[ \partial_i g ] = e^{-t} \| \partial_i g\|_1.\]
Hence, we have $ e^{-t} \| \partial_i  g \|_1 \le t^{ - (1-p)/2}  \|\partial_i P_t g \|_p^p$, as desired.
  \end{proof}

\subsection{An Alternative Refined Gaussian FKG Inequality}


In order to prove Theorem~\ref{thm:alt_bound_fn}, we need the
following identity for the covariance of a pair of functions
w.r.t.\ the Gaussian measure, which follows from~\cite[Lemma
3.3]{Sourav} using the polarization identity: $2\cov(f, g) =
\var(f+g)  - \var(f) - \var(g)$.
\begin{proposition} \label{prop:FKG_error}
Let $f, g: \Real^n \to \Real$ be two absolutely continuous
functions and suppose that $ \| \nabla f\|_2^2 ,  \| \nabla
g\|_2^2 \in L^2 (\mu^{\otimes n})$. Then
\begin{equation}\label{FKG_error}
\E_\mu[fg]  - \E_\mu[f] \E_\mu[g] =  \sum_{i=1}^n \int_0^{\infty}
e^{-t}\E_\mu \big[ \partial_i f P_{t} \partial_i g \big]  dt.
 \end{equation}
\end{proposition}

Note that if $f,g$ are increasing, then the RHS is clearly
non-negative, and hence, $\E_\mu[fg]  - \E_\mu[f] \E_\mu[g] \geq
0$.  This already implies the Gaussian FKG inequality~\cite{FKG}.
Moreover, the proposition gives a precise expression for $\cov(f,
g)$. However, as the precise expression is not so convenient to
work with, we replace it by a more convenient lower bound to
obtain Theorem~\ref{thm:alt_bound_fn}.

\begin{proof}[Proof of Theorem~\ref{thm:alt_bound_fn}]
First of all, note that since $f$ are increasing and $|f| \le 1$,
we have   $ \int_{\Real^n} \partial_i f(z) dz \le 2$ for all $i$.
Hence, $\E_\mu[\partial_i f] \le 1$. The same conclusion also
holds for $g$.

To prove the theorem, we will use Borell's reverse
hypercontractive inequality~\cite{Borell} which  implies the
following result. (See Corollary 3.3
of~\cite{Mossel-Inverse} for a discrete version of the result. The
Gaussian version presented here follows immediately by a CLT
argument.) Let $f_1, f_2 : \Real^n \to \Real_+$ be smooth bounded
functions, then for any $p, q \in (0, 1)$ such that $e^{-2t} \le
(1-p)(1-q)$, the following inequality holds:
\begin{equation} \label{eq: rev_hyp}
 \E_\mu [f_1P_t f_2]  \ge \| f_1\|_p \|f_2 \|_q.
 \end{equation}
 Here the norms are taken w.r.t.\ the Gaussian measure $\mu^{\otimes n}$.
Fix $1 \le i \le n$. Using \eqref{OU_fact:1} and the fact that
$P_t$ is reversible w.r.t.\ $\mu^{\otimes n}$, we have
\begin{align}
\int_0^{\infty} e^{-t}\E_\mu \big[ \partial_i f P_{t} \partial_i g \big]  dt &\ge \int_1^{\infty} e^{-(t-1)}\E_\mu \big[ \partial_i P_{1/2} f P_{t-1} \partial_i P_{1/2}g \big]  dt \notag\\
&=  \int_0^{\infty} e^{-t}\E_\mu \big[ \partial_i P_{1/2} f P_{t} \partial_i P_{1/2}g \big]  dt \notag\\
&=  \int_0^{1} \E_\mu \big[ \partial_i P_{1/2} f T_s \partial_i P_{1/2}g \big]  dt \quad [ T_s := P_{\log(1/s)}] \label{ineq:intergal1}.
\end{align}
By \eqref{eq: rev_hyp} and Lemma~\ref{eq:p_norm_bound}, we deduce that
\begin{equation}\label{eq:apply_rev_hyp}
\E_\mu \big[ \partial_i P_{1/2} f T_{s} \partial_i P_{1/2}g \big]  \ge \|\partial_i P_{1/2} f \|_p \|\partial_i P_{1/2} g \|_q \ge (2 e)^{ -( \frac{1}{2p} + \frac{1}{2q})} \|\partial_i  f \|_1^{1/p} \|\partial_i g \|_1^{1/q},
\end{equation}
for $s>0$ such that $s^2 \le (1-p)(1-q)$. Optimizing the RHS of
\eqref{eq:apply_rev_hyp}  over $p, q \in (0, 1)$ satisfying $ s^2
\le (1-p)(1-q)$, we obtain
\[ \E_\mu \big[ \partial_i P_{1/2} f T_{s} \partial_i P_{1/2}g \big]  \ge \exp \left( - \frac12 \frac{a_i^2 +2s a_i b_i+ b_i^2}{1-s^2}\right), \]
where $a_i, b_i>0$ are such that $(2e)^{-1/2} \|\partial_i  f \|_1 = e^{-a_i^2/2}$ and $(2e)^{-1/2} \|\partial_i  g\|_1  = e^{-b_i^2/2}$. Hence, by \eqref{ineq:intergal1},
\begin{align}
\int_0^{\infty} e^{-t}\E_\mu \big[ \partial_i f P_{t} \partial_i g \big]  dt &\ge \int_0^1  \exp \left( - \frac12 \frac{a_i^2 +2s a_i b_i+ b_i^2}{1-s^2}\right) ds \notag\\
&\ge  \eps  \exp \left( - \frac12 \frac{a_i^2 +2 \eps a_i b_i+ b_i^2}{1-\eps^2}\right), \label{ineq:lbound1}
\end{align}
for any $\eps \in (0,1)$.  We are interested in finding a lower bound of the RHS of \eqref{ineq:lbound1} when $a_i$ and $b_i$ are large.
Note that the derivative of the RHS  of \eqref{ineq:lbound1} as a function of $\eps$ vanishes approximately at $\eps \approx 1/a_ib_i$.  Plugging in $\eps = 1/a_ib_i$ in \eqref{ineq:lbound1}, we obtain
\begin{align}  \frac{1}{a_ib_i}   &\exp \left( - \frac12 \frac{a_i^2 +2 \eps a_i b_i+ b_i^2}{1-\eps^2}\right) \ge   \frac{1}{a_ib_i}   \exp \left( - \frac12  ( a_i^2 +2 \eps a_i b_i+ b_i^2)(1+\eps^2 + O(\eps^4))\right) \notag \\
&\ge   \frac{c_1}{a_i b_i}  e^{-(a_i^2+b_i^2)/2} \ge c_1 e \cdot
\frac{ \|\partial_i f\|_1 \|\partial_i g\|_1} {\sqrt{\log
(e/\|\partial_i f\|_1) \log (e/\|\partial_i g\|_1 )}},
\label{ineq:final}
\end{align}
where $c_1>0$ is a universal constant.  In the second inequality
above, we used the fact $\eps^2 (a_i^2 + b_i^2)  = O(1)$. This is
because  $a_i , b_i$ are bounded from below, which follows from
the fact  $\|\partial_i  f \|_1, \|\partial_i  g \|_1  \le 1$. Now we conclude the proof
by combining Proposition~\ref{prop:FKG_error} and the bounds
\eqref{ineq:lbound1} and \eqref{ineq:final} and by taking $c = c_1
e$.
\end{proof}

\subsection{A direct approach towards inverse Gaussian BKS}


In this subsection we aim to prove a  Gaussian analogue of the
inverse BKS theorem (see Proposition~1.3 of~\cite{BKS}).

\begin{proposition}\label{Cor:Converse-new}
Let  $f :\Real^{n} \to \Real$ be a continuously differentiable increasing function such that $ \| \nabla f\|_2^2  \in L^2 (\mu^{\otimes n})$. Then for any $\rho \in (0,1)$,
\[ \vg(f, \rho) \ge (1-\rho^2) \sum_{i=1}^n \E_\mu [\partial_i  f] ^2. \]
\end{proposition}
\begin{proof}
First we show that for  if $f$ is as given in the proposition, then
\begin{equation}\label{Eq:Converse-new1}
\var_\mu(f) \ge  \sum_{i=1}^n \E_\mu [\partial_i f] ^2.
\end{equation}
Taking $f=g$ in Proposition~\ref{prop:FKG_error}, we have
\begin{equation}\label{Eq:Converse-new2}
\var_\mu(f) =  \sum_{i=1}^n  \int_0^\infty e^{-t} \E_\mu[ \partial_i f P_t \partial_i f] dt.
\end{equation}
We claim that $\E_\mu[ \partial_i f P_t \partial_i f] $ is a nonincreasing function of $t$.  Indeed,
\[
\frac{d}{dt} \E_\mu[g P_t g]  =   \E_\mu[ g P_t  \sum_{j=1}^n ( \partial_j^2 g  -x_j \partial_j g)] =
\sum_{j=1}^n \E_\mu[ g_t  \partial_j^2 g  -   x_j g_t \partial_j g],
\]
where $g_t = P_t g$. Integration by parts yields
\[
\E_\mu[ g_t  \partial_j^2 g  -   x_j g_t \partial_j g] =
\E_\mu[ g  \partial_j^2 g  -   \partial_j ( g_t \partial_j g) ]  =
- \E_\mu[ \partial_j g_t \partial_j g] = - e^{-t} \E_\mu[\partial_j g P_t \partial_j g] \le 0,
\]
and hence,
\[
\frac{d}{dt} \E_\mu[g P_t g] = \sum_{j=1}^n \E_\mu[ g_t  \partial_j^2 g  -   x_j g_t \partial_j g] \leq 0.
\]
Therefore,
\begin{equation}\label{Eq:Converse-new3}
\E_\mu[ \partial_i f P_t \partial_i f]  \ge \E_\mu[ \partial_i f P_\infty \partial_i f]  =  \E_\mu[ \partial_i f ]^2.
\end{equation}
Combination of~\eqref{Eq:Converse-new2} with~\eqref{Eq:Converse-new3} yields~\eqref{Eq:Converse-new1}.

By~\eqref{Eq:Converse-new1} and~\eqref{OU_fact:1},
\[
\var_\mu(P_t f) \ge  \sum_{i=1}^n \E_\mu [\partial_i P_t f] ^2 =
\sum_{i=1}^n e^{-2t} \E_\mu [P_t \partial_i  f] ^2 =   e^{-2t} \sum_{i=1}^n  \E_\mu [\partial_i  f] ^2.
\]
Note that by the definition of the Orenstein-Uhlenbeck operator, we have
\begin{equation} \label{eq:OU_ns}
\vg(f,\rho) = \E_\mu [ f P_t f ] - \E_\mu [f]^2 = \var_\mu(P_t f).
\end{equation}
This completes the proof.
\end{proof}
As a corollary (which we will prove in the next section), we
obtain an inverse Gaussian BKS theorem for increasing functions.
\begin{corollary}\label{cor:gaussian_bks_inv}
Let  $\{A_\ell \subseteq \Real^{n_{\ell}}\}$ be a sequence of
increasing sets. If $\{A_\ell\}$ is asymptotically Gaussian noise
sensitive, then $\sum_{i=1}^{n_{\ell}} \ig_i(A_\ell)^2 \to 0$ as
$\ell \to \infty$.
\end{corollary}

\section{Smooth approximation of characteristic functions of monotone sets}\label{smooth_approx}
In this section, we prove a result that connects the partial derivative of  the characteristic function  of an increasing set after being smoothed by the action of Ornstein-Uhlenbeck operator $P_t$   to its geometric influence as $t \downarrow 0$.  This will help us in deriving various theorems presented in the introduction, which involve sets,   from the respective theorems involving  $C^1$ functions.

 Recall that as defined in the
introduction, for any set $A \subseteq \mathbb{R}^n$, for each $ 1
\leq i \leq n$ and an element $x =(x_1, x_2, \ldots, x_n) \in
\mathbb{R}^n$, the restriction of $A$ along the fiber of $x$ in
the $i$-th direction is given by
\[
A^{x}_i := \{ y \in \mathbb{R}: (x_1, \ldots,  x_{i-1}, y,
x_{i+1}, \ldots, x_n) \in A \}.
\]
\begin{definition}\label{def:increasing}
A set $A \subset \mathbb{R}^n$ is called increasing (decreasing)
if its characteristic function $\1_A$ is an increasing
(decreasing) function in each coordinate. For any increasing set $A \subset \Real^n$
and for any $x \in \Real^n$, define
\[
t_i(A;x^{(-i)}) := \inf \{ y : y \in A^x_i\} \in [- \infty,  \infty],
\]
where $x^{(-i)} = (x_1, x_2, \ldots, x_{i-1}, x_{i+1}, \ldots,
x_n) \in  \Real^{n-1}$ and we use the convention that the infimum
of the empty set is $+\infty$.
\end{definition}
Note that $t_i(A; \cdot)$ is a decreasing function of $x^{(-i)} $
for any increasing set $A$. Also, for an increasing set $A$,  its geometric influence is given by $\ig_i(A)= \int_{ \Real^{n-1}} \phi(t_i(A;z^{(-i)})) \mu^{ \otimes n-1} ( dz^{(-i)})$.
%

\begin{lemma} \label{l:partial_deri_OU_geo_inf}
Let $A$ be a monotone subset of $\mathbb R^n$. Then, for each $i \in \{ 1, \ldots, n\}$,  we have
\[   \E_\mu[ \partial_ i P_t \1_A]  \to \ig_i(A) \quad \text{as }\  \  t \downarrow 0.\]
\end{lemma}
\begin{remark}
Lemma~\ref{l:partial_deri_OU_geo_inf} does not hold in general without the monotonicity assumption. For example, take $n=1$ and define  $A = \mathbb Q$, the set of rational numbers. Then $ P_t \1_A  = 0$ for any $t>0$ and hence  $ \lim_{ t \to 0+} \E_\mu[ \partial_ i P_t \1_A] = 0$  but it can be easily checked that  $\ig_1(A)  = \infty$.
\end{remark}

In order to prove Lemma~\ref{l:partial_deri_OU_geo_inf} we need the following standard lemma. For sake of completeness, we present its proof.

\begin{lemma} \label{lem:discon}
Let  $f: \Real^n \to \Real$ be a  monotone function. Then the set
of discontinuities of $f$ has Lebesgue measure zero.
\end{lemma}

\begin{proof}
The $n$-dimensional space $\mathbb{R}^n$ can be represented as a
disjoint union of straight lines $\bigcup_{\{z \in
\mathbb{R}^n:z_n=0\}} l_z$, where each line is defined as $l_z =
z+t(1,1,\ldots,1)$, $t \in \mathbb{R}$. We would like to show
that the set of discontinuities of $f$ on each line $l_z$ is of
Lebesgue measure zero, and then, the assertion of the lemma would
follow by a standard application of Fubini's theorem.

For each such line $l_z$, the restriction of $f$ to $l_z$ can be
represented by a one-dimensional function $f_z: \mathbb{R} \to
\mathbb{R}$ defined by $f_z(a)=f(z+(a,a,\ldots,a))$. Note that if
$f$ is not continuous at some $x \in \ell_z$,  then
\[
\lim_{\eps \to 0+}  \left(  \sup_{ y \in x + [-\eps, \eps]^n}
f(y) - \inf_{ y \in x + [-\eps, \eps]^n} f(y) \right) >0,
\]
which implies, by monotonicity of $f$, that
\[
\lim_{\eps \to 0+} f(x+(\eps,\ldots,
\eps))-f(x-(\eps,\ldots,\eps)) \neq  0.
\]
Hence, each discontinuity $x$ of $f$ corresponds to a
discontinuity of the one-dimensional function
$f_{z}$. 
Therefore, the set of discontinuities of $f$ on a line $l_z$ can
be embedded into the set of discontinuities of the function $f_z$.
However, for a fixed $z$, $f_z$ is a monotone function on the real
line, and thus, the set of its discontinuities is countable, and,
in particular, of Lebesgue measure zero. Thus, the set of
discontinuities of $f$ on each line $l_z$ is of Lebesgue measure
zero, which completes the proof.
\end{proof}

\begin{proof}[Proof of Lemma~\ref{l:partial_deri_OU_geo_inf}]
Without loss of generality, we assume that $A$ is an increasing
set. Let $W = (W_1, \ldots, W_n)$ be a standard Gaussian vector on
$\mathbb R^n$ and define $Y_t = e^{-t} x + \sqrt{ 1- e^{-2t}} W$.
Then we can write
\[P_t \1_A(x) = \E_\mu[ \1_A( Y_t)] = \E_\mu \left[ \bar \Phi\left(\frac{t_i(A; Y_t^{(-i)})  - e^{-t} x_i}{\sqrt{1- e^{-2t}}} \right) \right], \]
which, taking partial derivative w.r.t. $x_i$,  yields
\[ \partial_i P_t \1_A(x) = \E  \frac{e^{-t} }{ \sqrt{1 - e^{-2t}}} \phi \left(  \frac{t_i(A; Y_t^{(-i)})  - e^{-t} x_i}{\sqrt{1- e^{-2t}}} \right) \ge 0. \]
Therefore,
\begin{align} \label{eq:derivative_i}
 \E_\mu[ \partial_ i P_t \1_A]  &=  \int_{\mathbb R^n}  \E_W  \left [ \frac{e^{-t} }{ \sqrt{1 - e^{-2t}}} \phi \left(  \frac{t_i(A; Y_t^{(-i)})  - e^{-t} x_i}{\sqrt{1- e^{-2t}}} \right) \right]  \prod_{j=1}^n \phi(x_j) dx  \notag\\
 &=  \E_W  \int_{\mathbb R^n}  \phi(u) \phi \left(  \frac{\sqrt{1- e^{-2t}} u + t_i(A; Y_t^{(-i)})  }{e^{-t}} \right) du  \cdot  \prod_{j \ne i} \phi(x_j)  dx^{(-i)},
 \end{align}
 where in the last step we make a change of variable $ u =  \frac{ e^{-t} x_i  - t_i(A; Y_t^{(-i)})  }{\sqrt{1- e^{-2t}}}$.
Note that by Lemma~\ref{lem:discon}, we have $t_i(A; Y_t^{(-i)})
\to t_i(A; x^{(-i)})$ in distribution as $ t \to 0+$.  Hence,
taking limit as $t \to 0+$  in \eqref{eq:derivative_i}, we obtain,
by the Bounded Convergence Theorem,
 \[    \E_\mu[ \partial_ i P_t \1_A]   =  \int_{\mathbb R^n}   \phi \left(  t_i(A; x^{(-i)})  \right) du \cdot  \prod_{j \ne i} \phi(x_j)  dx^{(-i)}  = \ig_i(A). \]
  This completes the proof of the lemma.
%
\end{proof}
As a consequence of Lemma~\ref{l:partial_deri_OU_geo_inf},
Theorems~\ref{thm:gaussian_talagrand} and \ref{thm:alt_bound}
can now be easily derived from their functional counterparts.
\begin{proof}[Proof of Theorems~\ref{thm:gaussian_talagrand} and \ref{thm:alt_bound}]
For $t>0$, define $f_t = P_t \mathbf{1}_A$ and $g_t = P_t
\mathbf{1}_B$ for increasing sets $A, B$ of $\Real^n$. Note that $f_t$ and $g_t$ are increasing $C^\infty$ functions which are bounded by $1$. Thus we can apply Theorems~\ref{thm:alt_bound_fn} and  \ref{thm:gaussian_talagrand_func} with $f = f_t$ and $g_t$ and then let $t \to 0+$. In view of Lemma~\ref{l:partial_deri_OU_geo_inf}, the right hand sides of the inequalities  converge to appropriate quantities involving the geometric influences of the sets $A$ and $B$. Again by Lemma~\ref{l:partial_deri_OU_geo_inf}, $\1_A$ and $ \1_B$ are almost surely continuous,  hence   $f_t \to \1_A$ and   $g_t \to \1_B$ in probability. Therefore, $\E_\mu[ f_tg_t] - \E_\mu[ f_t] \E_\mu[g_t] \to \mu^{\otimes n} (A \cap B)  -\mu^{\otimes n} (A) \mu^{\otimes n} (B)$ by dominated convergence, which completes the proofs of the theorems.
\end{proof}

Note that above proof technique can not be immediately applied  to deduce Theorem~\ref{thm:Gaussian-Quantitative-BKS}  from Theorem~\ref{thm:Gaussian-Q-BKS_function} since Lemma~\ref{l:partial_deri_OU_geo_inf} does not hold for general non-monotone sets.  We overcome this obstacle by establishing a shifting lemma, which implies that it will be
sufficient to prove our theorem for  increasing sets. This shifting
lemma is a Gaussian analogue of Lemma 2.7 in~\cite{BKS}.
\begin{definition}
For $i \in \{1,2, \ldots, n\}$, the $i$-shift operator $M_i$
acting on subsets of $\Real^n$ is defined by:
\[
M_i(A) = \{ x \in \Real^{n} : x_i \ge {\bar \Phi}^{-1}(\mu(A^x_i))  \}.
\]
The shifting operator $M$ is defined as $M = M_1 \circ M_2 \circ
\ldots \circ M_n$.
\end{definition}

\begin{lemma}\label{lem:shifting}
Let $A \subseteq \Real^n$. For any $i \in \{1, 2, \ldots, n\}$ and
for any $\rho \in (0, 1)$, we have:
\begin{enumerate}
\item[(i)] $M (A)$ is  increasing.

\item[(ii)]$\mu^{\otimes n}(M(A)) =  \mu^{\otimes n}(A)$.

\item[(iii)] $\ig_i(M (A)) \le \ig_i(A)$.

\item[(iv)] $\zg(M(A), \rho) \ge \zg( A, \rho)$.

\end{enumerate}
\end{lemma}

\begin{proof}
The proofs of (i) and (ii) are standard (see \cite{Frankl}).

In order to prove (iii), we recall the notion of $h$-influences
defined in~\cite{keller11} and its relation to geometric
influences. For a function $h:[0,1] \to [0,1]$, the $h$-influence
of the $i$-th coordinate on $A$ (in the Gaussian space)  is defined as
\[
I^h_i(A) := \int h(\mu(A^{x}_i)) \mu^{\otimes n}(dx),
\]
where $A^{x}_i$ is the restriction of $A$ along the fiber of $x$
in the $i$-th direction. It was shown in previous work that:
\begin{itemize}
\item If the function $h$ is concave and continuous, then
$h$-influences of any set can only decrease under the action of
the shifting operator $M$ on that set (see Theorem 2.2 of
\cite{keller11}).

\item For $h(t) = \phi(\Phi^{-1}(t))$ (which is concave and
continuous), we have $\ig_i(A) \ge I_i^h(A)$ for any set $A$, and
$\ig_i(A) = I_i^h(A)$ for monotone increasing sets (see Lemmas 3.5
and 3.7 of~\cite{Geom-Influences}).
\end{itemize}
Combining these two facts, we have
\[
\ig_i(M (A)) =  I^h_i(M(A)) \le  I^h_i(A) \le \ig_i(A),
\]
as asserted in (iii).

To prove (iv),  it is sufficient to show that $\zg(M_j(A), \rho)
\ge \zg( A, \rho)$ for each $j \in \{1,2,\ldots, n\}$. Let $W, W'$
be two i.i.d.\ standard Gaussian vectors on $\Real^n$ and set $\rW
= \sqrt{1 - \rho^2} W + \rho W'$ (as defined above). We have
\begin{eqnarray}\label{Eq:Aux1}
&\zg( A, \rho)  = \E[ \1_{A^{W}_j}(W_j)  \1_{A^{\rW}_j}(\rW_j) ]  \notag\\
& =
\E_{W^{(-j)}, {W'}^{(-j)}} \left [ \E \big[ \1_{A^{W}_j}(W_j)
\1_{A^{\rW}_j}(\rW_j) \big| W^{(-j)}, {W'}^{(-j)} \big] \right].
\end{eqnarray}
By Borell's isoperimetric inequality~\cite{Borell2}, amongst all
pairs of subsets $S, T$ of the real line such that $\mu(S) = a$
and $\mu(T) = b$, the joint probability $\prob[ W_j \in S, \rW_j
\in T ]$ is maximized when $S = [\bar \Phi^{-1}(a), \infty)$ and
$T = [\bar \Phi^{-1}(b), \infty)$. This implies that
\begin{equation}\label{Eq:Aux2}
\E [ \1_{A^{x}_j}(W_j)  \1_{A^{y}_j}(\rW_j) ] \le
\E [ \1_{ M_j(A)^{x}_j}(W_j)  \1_{M_j(A)^{y}_j}(\rW_j) ]
\quad \forall x, y \in \Real^n.
\end{equation}
Assertion (iv) follows immediately by plugging
Equation~(\ref{Eq:Aux2}) into Equation~(\ref{Eq:Aux1}).
\end{proof}

\begin{proof}[Proof of
Theorem~\ref{thm:Gaussian-Quantitative-BKS}] By
Lemma~\ref{lem:shifting}, it is sufficient to prove the theorem
for increasing sets. Now we can follow the proof of
Theorems~\ref{thm:gaussian_talagrand} and \ref{thm:alt_bound}
to complete the proof. We omit the details.
\end{proof}

We point out that a Gaussian analogue of  the original BKS theorem  follows immediately from
Theorem~\ref{thm:Gaussian-Quantitative-BKS}.

\begin{corollary}
Let $A_\ell \subseteq \Real^{n_\ell}$ be a sequence of sets and
suppose that $\sum_{i=1}^{n_\ell} \ig_i(B)^2 \to 0$ as $\ell \to
\infty$. Then $\{ A_\ell\}$ is asymptotically Gaussian
noise-sensitive.
\end{corollary}

\begin{proof}[Proof of Corollary~\ref{cor:gaussian_bks_inv}]
Again, we can follow the proof of
Theorems~\ref{thm:gaussian_talagrand} and \ref{thm:alt_bound}
to show that, for any increasing set $A \subseteq \Real^n$,
\[ \vg(A, \rho) \ge (1-\rho^2) \sum_{i=1}^n \ig_i(A)^2. \]
The assertion of the corollary follows immediately.
\end{proof}

\subsection{Comparison Between Theorems~\ref{thm:gaussian_talagrand}
and~\ref{thm:alt_bound}}
\label{sec:sub:tightness}

Let us compare the performances of
Theorems~\ref{thm:gaussian_talagrand}  and \ref{thm:alt_bound} in two important special cases.
\begin{itemize}
\item \textbf{Threshold sets in $\Real^n$.} Let $A = \{ x \in \Real^n:
n^{-1/2}\sum_{i=1}^{n} x_i > - t \}$ and $B = \{ x \in \Real^n:
n^{-1/2}\sum_{i=1}^{n} x_i > t\}$. In this case, $\mu(A)  = \eps$
and $\mu(B) = 1- \eps$ where $\eps = \Phi^{-1}(-t)$, and hence,
$\mu^{\otimes n}(A\cap B)  -  \mu^{\otimes n}(A) \mu^{\otimes
n}(B) = \eps^2$. It is easy to show that $\ig_i(A) = \ig_i(B)
\asymp n^{-1/2}\eps \sqrt{\log(1/\eps)}$ for each $i$. Thus,
Theorem~\ref{thm:gaussian_talagrand} gives a lower bound of
order $\eps^2$  whereas Theorem~\ref{thm:alt_bound} yields a
lower bound of order $\eps^2/\log n$. Therefore, in this example,
Theorem~\ref{thm:gaussian_talagrand} is  tight as $t \to \infty$
($\eps \to 0$) up to  a constant factor for any $n$, while
Theorem~\ref{thm:alt_bound} is off by a factor of $\log n$.

\item \textbf{Sets that depend on a single coordinate.} Let $n=1$
(which is equivalent to the case when both sets depend on a single
coordinate). In this case, Theorem~\ref{thm:alt_bound} is
strictly stronger than Theorem~\ref{thm:gaussian_talagrand}.
Indeed, for $n=1$, the bounds given by the theorems are (up to a
constant):
\[
 \frac{ \ig(A)\ig(B)} {\sqrt{\log (e/\ig(A)) \log (e/\ig(B) )}}
\qquad \mbox{ and } \qquad \frac{ \ig(A)\ig(B)} {\log (e/\ig(A)
\ig(B))}.
\]
Since
\[
\log(e/\ig(A) \ig(B)) \geq \frac{1}{2} \left(\log(e/\ig(A)) +
\log(e/\ig(B) \right),
\]
the left bound is always greater then the right one by the
inequality between the arithmetic and geometric means. Moreover,
it can be shown that the bound of Theorem~\ref{thm:alt_bound} is
asymptotically tight for any choice of the sets $A,B$, while
Theorem~\ref{thm:gaussian_talagrand} is not tight for
$A=(-t,\infty)$, $B=(e^t,\infty)$ as $t \to \infty$.
\end{itemize}


\section{Other Probability Spaces}
\label{sec:Discrete}

In this section, we show how one can use the Gaussian Talagrand
bounds  obtained in the previous sections to prove analogous
bounds for other product spaces, including all discrete product
spaces, the space $[0,1]^n$ endowed with the Lebesgue measure,
etc. Next we will deduce a BKS theorem for the product biased
measure on the discrete cube $\{-1, 1\}^n$ from its Gaussian
counterpart.  We should mention here that it is not clear if it is
possible to find a reduction from the Gaussian  BKS theorem to an
analogous BKS theorem for a general  discrete product space
$([q]^n, \gamma^{\otimes n})$. Indeed, while the Ornstein-Uhlenbeck
semigroup action is same as adding `small' amount of  noise to
every coordinate,  the standard noise operator (on a discrete
product space) amounts to adding `big' noise to a small number of
coordinates. That they are equivalent is far from obvious.

Since there is no single natural definition of influences for such spaces, we formulate
the results in terms of the $h$-influences defined in~\cite{keller11} (which turns out
to be the most natural way to state them), and then mention the formulation with
respect to more common definitions of influences. First we recall the definition of $h$-influences.
\begin{definition}
Let $\Omega$ be a probability space endowed with a probability measure $\gamma$.
For a function $h:[0,1] \to [0,1]$, the $h$-influence of the $i$-th coordinate on a set $A$ in the product space $(\Omega^n, \gamma^{\otimes n})$ is defined as
\[
I^h_i(A) := \mathbb{E}_{\gamma} [ h(\gamma(A^{x}_i))],
\]
where $A^{x}_i$ is the restriction of $A$ along the fiber of $x$
in the $i$-th direction and $\E_\gamma$, as always,  denotes the expectation w.r.t.\ the product measure $\gamma^{\otimes n}$.
\end{definition}

Throughout this section, we consider $h$-influences with respect
to the function $h(t)=\phi(\Phi^{-1}(t))$. For sake of simplicity,
we formulate the results for discrete probability spaces. The
results for other spaces, such as the space $[0,1]^n$ endowed with
the Lebesgue measure, can be derived similarly.

For $q>1$, let $[q]=\{1,2,\ldots,q\}$, and let $\gamma$ be a
probability measure on $[q]$. Without loss of generality, we
assume that $\gamma(i) >0$ for all $i \in [q]$ and denote the
smallest atom in $([q],\gamma)$ by $\alpha = \min_{ i \in [q]}
\gamma(i)$. In order to obtain the reduction from
$([q]^n,\gamma^{\otimes n})$ to $(\Real^n,\mu^{\otimes n})$, we
define $\psi : \Real \to [q]$ to be an increasing function such
that the push forward $\mu \circ \psi^{-1} $ has law $\gamma$. For
example, $\psi(u) = \min\{i \in [q]: F(i) > \Phi(u) \}$, where $F$
is the distribution function of $\gamma$. Define $\psi^{\otimes
n}: \Real^n \to [q]^n$ by $\psi^{\otimes n} (u_1, \ldots, u_n) =
(\psi(u_1), \ldots, \psi(u_n))$, and set $A_G := (\psi^{\otimes
n})^{-1}(A)$.

Obviously, $\mu^{ \otimes n}(A_G) = \gamma^{ \otimes n}(A)$ for any $A \subseteq [q]^n$. Moreover, a similar equality holds with respect to the restriction along fibers: If $u \in \Real^n$ such that $\psi^{\otimes n}(u) =J \in [q]^n$, then  the fibers $(A_G)^{u}_i$ and $A^J_i$ satisfy:
 \[ (A_G)^{u}_i = \psi^{-1} (A^J_i).\]
 Consequently, $\mu((A_G)^u_i ) = \mu(\psi^{-1} (A^J_i)) = \gamma(A^J_i)$. This allows us to relate
the geometric influences of $A_G$ to the $h$-influences of $A$. Indeed, it was shown
in~\cite{Geom-Influences} that for $h(t) = \phi(\Phi^{-1}(t))$, we have $\ig_i(B) \ge I_i^h(B)$ for any set $B \subseteq \Real^n$, and $\ig_i(B) = I_i^h(B)$ for monotone increasing sets (see Lemmas 3.5 and 3.7 of~\cite{Geom-Influences}). Hence, for any
$A \subseteq [q]^n$ and for any $1 \leq i \leq n$,
\begin{equation}\label{Eq5.1-new}
\ig_i(A_G) \geq I^h_i(A).
\end{equation}

This allows us to obtain analogues of Gaussian correlation bounds
for the product space $([q]^n, \gamma^{\otimes n})$.
\begin{theorem}\label{Thm:Discrete1}
Let $A, B$ be two increasing subsets of $[q]^n$. Then,
\[\begin{split}  &\gamma^{\otimes n}(A \cap B)  - \gamma^{\otimes n}(A)\gamma^{\otimes n}(B) \ge \\
 &c \max \left( \sum_{i=1}^n \frac{ I_i^h(A)I_i^h(B)} {\sqrt{\log (1/I_i^h(A)) \log (1/I_i^h(B) )}},
\varphi \Big(\sum_{ i=1}^n I_i^h(A)  I_i^h(B)\Big) \right),
\end{split}\]
where $c>0$ is a universal constant, $h(t)=\phi(\Phi^{-1}(t))$,
and $\varphi(x) = x/ \log(e/x)$.
\end{theorem}

\begin{proof}
Since the functions $x \mapsto x /\sqrt{\log(1/x)}$ and $x \mapsto
x / \log(e/x)$ are increasing in $(0,1)$, the assertion follows by
applying Theorem~\ref{thm:gaussian_talagrand}
and Theorem~\ref{thm:alt_bound} to the increasing sets $A_G,B_G \subseteq \Real^n$
coupled with the observation~\eqref{Eq5.1-new}.
\end{proof}

An interesting special case is the discrete cube $\{-1,1\}^n$
endowed with the product biased measure $\nu_\ga^{\otimes n}$, where  $\nu_\ga = \ga \delta_{1} + (1-\ga)
\delta_{-1}$  (w.l.o.g.\ for $0<\ga<1/2$). In this
case, the $h$-influence with $h(t)=\phi(\Phi^{-1}(t))$ satisfies
\[
I^h_i(A) = \ga \sqrt{\log(1/\ga)} I_i(A),
\]
where $I_i(A)$ is defined similarly to \eqref{def:inf} (but
instead of taking the expectation w.r.t.\ the uniform measure
$\nu^{\otimes n}$,  we use the product biased measure
$\nu_\alpha^{\otimes n}$). Hence, Theorem~\ref{Thm:Discrete1}
gives the bound
\[
\nu_\ga^{\otimes n}(A \cap B)- \nu_\ga^{\otimes n}(A)\nu_\ga^{\otimes n}(B) \geq c \varphi \left(\ga^2
\log(1/\ga) \sum_{i=1}^n I_i(A) I_i(B) \right),
\]
which was already shown in~\cite[Proposition
3.12]{Keller-Reduction}. We note that unlike the result
of~\cite{Keller-Reduction}, in Theorem~\ref{Thm:Discrete1} the
$h$-influences in the RHS appear without a ``scaling factor''
depending on $\ga$. This shows that in some sense, this
$h$-influence, which is the discrete variant of the geometric
influence, is more natural than the definition of influence used
in~\cite{Keller-Reduction} for the biased measure.

\medskip

In order to obtain an analogue of
Theorem~\ref{thm:Gaussian-Quantitative-BKS} for the biased cube $(\{-1,1\}^n, \nu_\ga^{\otimes n}) $, we need to find the
exact relation between Gaussian noise sensitivity and discrete
noise sensitivity (as defined in the introduction but now both $X$ and $X^\eta$ are distributed (marginally)  as $\nu_\alpha^{\otimes n}$).
\begin{lemma}\label{Lemma:5.1-new}
Consider the probability space $(\{-1, 1\}^n, \nu_\ga^{\otimes n})$. Let $A$ be a subset of $\{-1, +1\}^n$  and let $A_G$ be as defined above. Then
for any $\rho \in (0,1)$
\[
\vg(A_G,\rho) = \mathrm{VAR}(A,\eta),
\]
for $\eta = \eta(\rho,\alpha) = \frac{\prob[ W_1< \Phi^{-1}(\ga), W^\rho_1> \Phi^{-1}(\ga)]}{\ga(1-\ga)} $, where $(W_1, W^\rho_1)$ is a bivariate normal random vector with mean zero, unit  variance and correlation $\sqrt{1-\rho^2}$.
\end{lemma}

\begin{proof}
Let $X$ and $X^\eta$ be two $(1- \eta)$-correlated vectors on
$(\{-1, +1\}^n, \nu_\ga^{\otimes n})$ and let $(W, W^\rho)$ be
Gaussian vectors on $\mathbb R^n$ as defined in
Definition~\ref{def:GNS}.  Clearly, $\mu^{\otimes n}(A_G) =
\nu_\ga^{\otimes n}(A)$. To equate $Z(A, \eta)$ to $Z^{\mathcal
G}(A_G, \rho)$,  we want to choose $\rho>0$ such that the random
vectors $(X, X^\eta)$ and $\big(\psi^{\otimes n} (W),
\psi^{\otimes n} (W^\rho)\big)$ have the same distributions on
$\{-1, 1\}^n \times \{-1, 1\}^n$. Note that this is equivalent to
the condition
\[ \prob[X_1 =  - 1, X^\eta_1 = -1] = \prob[ W_1< \Phi^{-1}(\ga), W^\rho_1< \Phi^{-1}(\ga)], \]
which is same as
\[ \alpha  - \alpha(1- \alpha) \eta = \prob[ W_1< \Phi^{-1}(\ga), W^\rho_1< \Phi^{-1}(\ga)]. \]
The lemma now follows immediately.
\end{proof}

\begin{theorem}\label{Thm:Discrete2} Consider the product space $(\{-1,1\}^n, \nu_\ga^{\otimes n})$.
For any $n$, for any set $A \subset \{-1,1\}^n$, and for any $\eta \in
(0, 1)$,
\[
\mathrm{VAR}(A,\eta) \leq C_1 \cdot \left( \sum_{i=1}^n I^h_i(A)^2
\right)^{C_2 \rho^2},
\]
where $h(t)=\phi(\Phi^{-1}(t))$, $\rho$ is as defined in
Lemma~\ref{Lemma:5.1-new}, and $C_1,C_2>0$ are universal constants.
\end{theorem}

\begin{proof}
Consider the set $A_G$ defined as above, and the corresponding
``monotonized'' set $M(A_G)$ (see Lemma~\ref{lem:shifting} above).
By Lemma~\ref{lem:shifting}(iv) and Lemma~\ref{Lemma:5.1-new},
\begin{equation}\label{Eq:5.2-new}
\vg(M(A_G),\rho) \ge \vg(A_G,\rho) = \mathrm{VAR}(A,\eta).
\end{equation}
On the other hand, by properties of the monotonization operator
$M$, we have:
\begin{equation}\label{Eq:5.3-new}
\ig_i(M(A_G)) = I^h_i(M(A_G)) \le I^h_i(A_G) = I^h_i(A)
\end{equation}
(see the proof of Lemma~\ref{lem:shifting}(iii) above). Applying
Corollary~\ref{thm:Gaussian-Quantitative-BKS} to the set $M(A_G)$,
we get:
\begin{equation}\label{Eq:5.4-new}
\vg(M(A_G),\rho) \leq C_1 \cdot \left( \sum_{i=1}^n \ig(M(A_G))^2
\right)^{C_2 \rho^2}.
\end{equation}
Combination of~\eqref{Eq:5.4-new} with~\eqref{Eq:5.2-new}
and~\eqref{Eq:5.3-new} yields the assertion.
\end{proof}
Let's  compare the above bound to the following bound obtained
 in~\cite[Theorem
7]{Keller-Kindler} in the regime when $\eta>0$ is small but fixed
and $\alpha \to 0$:
\[
\mathrm{VAR}(A,\eta) \leq c'_1 \cdot \left( \alpha(1-\alpha)
\sum_{i=1}^n I_i(A)^2 \right)^{\beta(\eta, \alpha) \cdot \eta},
\]
where $\beta(\eta, \alpha)\cdot \eta \asymp_\eta
1/\log(1/\alpha)$. Note that after switching back to ordinary
influences, Theorem~\ref{Thm:Discrete2} reads:
\begin{equation}\label{discrete_bkkkl}
\mathrm{VAR}(A,\eta) \leq C_1 \cdot \left( \alpha^2 \log(1/\alpha)
\sum_{i=1}^n I_i(A)^2 \right)^{C_2 \cdot \rho(\eta, \alpha)^2},
\end{equation}
We are interested in finding a reasonable lower bound (up to a constant that may depend on $\eta$) on $\rho(\eta, \alpha)$.
Set $t  = \bar \Phi^{-1}(\alpha) \asymp \sqrt{\log(1/\alpha)}$. Note that,
\begin{align*}  \prob[ W_1>t , W^\rho_1<  t] &\le \prob[ W_1>t] \prob[ t \sqrt{1- \rho^2} + \rho W_1' <  t]\\
 &= \alpha  \prob \left[  W_1' <  \frac{ t(1- \sqrt{1- \rho^2})}{\rho} \right]   \le \alpha \prob [  W_1' <  t \rho].
 \end{align*}
Since both $ \prob[ W_1>t , W^\rho_1<  t]$ and $ \alpha \prob [
W_1' <  t \rho]$ are increasing functions of $\rho$,  a lower
bound on $\rho(\eta, \alpha)$ can be achieved by solving $\eta
\alpha(1 - \alpha) =  \alpha \prob [  W_1' <  t \rho]$, which
yields $t \rho \asymp_{\eta} 1$, or, $\rho^2 \asymp_\eta 1/
\log(1/\alpha)$.  So, the asymptotic performance of
Theorem~\ref{Thm:Discrete2} matches with that of
\cite{Keller-Kindler} (which was shown in~\cite{Keller-Kindler} to
be essentially tight).
\medskip

We now relate our results to more common definitions of influences
in the product spaces $([q]^n, \gamma^{\otimes n})$.

\medskip

\noindent \textbf{Variance Influence.} This notion, used e.g.
in~\cite{hatami09,Maj-Stablest}, is defined as:
\[
\iv_i(A) := \mathbb{E}_{\gamma} [ \mathrm{Var}(\1_{A^{x}_i})].
\]
It is clear that the variance influence coincides with the
$h$-influence for $h(t)=t(1-t)$, and hence, it is always smaller
(up to a constant factor) than the $h$-influence with
$h(t)=\phi(\Phi^{-1}(t))$. Hence, Theorem~\ref{Thm:Discrete1}
holds without change for the variance influences. In order to find
a lower bound of variance influence in terms of $h$-influence, we
consider the contribution of a single fiber to $I^h_i(A)$ and to
$\iv_i(A)$. If $\mu(A^{x}_i)=t \le 1/2$, then these contributions
are $t \sqrt{\log(1/t)}$ and $t(1-t)$, respectively. Note that if
$\alpha$ is the size of the smallest atom in $([q],\gamma)$, then
either $t \in \{0,1\}$ or $t \in [\alpha, 1-\alpha]$. In both
cases,
\[
\frac{t \sqrt{\log(1/t)}}{t(1-t)} \leq 2\sqrt{\log(1/\alpha)}.
\]
Hence, for any $A$ and  $i$,
\[
I^h_i(A) \leq 2\sqrt{\log(1/\alpha)} \iv_i(A),
\]
and thus, Theorem~\ref{Thm:Discrete2} holds (for $q=2$)  with $4\log(1/\alpha)
\sum_i \iv_i(A)^2$ in place of $\sum_i I^h_i(A)^2$.

\medskip

\noindent \textbf{BKKKL Influence.} This influence, used
in~\cite{bourgain92,hatami09} is given by:
\[
I^{\BKKKL}_i(A) := \mathbb{E}_{\gamma} [ h(\mu(A^{x}_i))],
\]
where $h(t)=1$ if $t \in (0,1)$, and $h(t)=0$ if $t \in \{0,1\}$. This definition coincides with $I_i(A)$ for $q=2$ and we have already seen in \eqref{discrete_bkkkl} how Theorem~\ref{Thm:Discrete2} should look in this case.
As for Theorem~\ref{Thm:Discrete1}, since the
contribution of each fiber to $I^h_i(A)$ is either zero or at
least $\alpha \sqrt{\log(1/\alpha)}$, it follows that the theorem
holds with $\alpha^2 \log(1/\alpha) I_i^{\BKKKL}(A)I_i^{\BKKKL}(B)$
instead of $I_i^h(A) I_i^h(B)$.

\section{Open Problems}
\label{sec:Open}

We conclude the paper with a few directions for further research
suggested by our results and by recent related work.
\begin{enumerate}
\item The first issue left open in this paper is to prove a
quantitative BKS theorem for all other discrete product spaces. In
fact, we weren't able to deduce it by a reduction from the
Gaussian version even for the simplest case $([q]^n,
\lambda^{\otimes n})$ where $q>2$ and $\lambda$ is the uniform
measure on $[q]$. We note that we have a direct proof of
quantitative BKS for all discrete spaces, using a generalization
of the techniques used in~\cite{Keller-Kindler}, along with
hypercontractive estimates for general discrete measures obtained
by Wolff~\cite{Wolff}. However, the proof is cumbersome and the
result is not tight, and hence, a reduction from the Gaussian case
is more desirable.

\item It would be interesting to find alternative ``direct''
proofs of Theorems~\ref{thm:gaussian_talagrand}
and~\ref{thm:Gaussian-Quantitative-BKS}, which do not rely on
their counterparts on the discrete cube. In particular, we wonder
whether one can combine the reverse hypercontractivity technique
used in the proof of Theorem~\ref{thm:alt_bound} with the the
classical hypercontractivity used in the proof of
Theorem~\ref{thm:gaussian_talagrand} to obtain a new lower bound
that will enjoy the benefits of both theorems.

\item Probably the most interesting direction is to find
applications of the results. Both Talagrand's lower bound and the
BKS theorem have various applications, and even the recent
generalization of the BKS theorem to biased
measures~\cite{Keller-Kindler} was already applied to percolation
theory~\cite{ABGM}. On the other hand, Gaussian noise sensitivity
was recently studied by Kindler and o'Donnell~\cite{Guy-Ryan} and
used to obtain applications to isoperimetric inequalities and to
hardness of approximation. Hence, it will be interesting to find
also applications of Talagrand's lower bound or of the BKS theorem
in the Gaussian setting.

\item Finally, our understanding of influences in product spaces
is still very far from complete. In particular, only a very few is
known about influences with respect to non-product measures, and
it is even unclear what should be the natural definition of
influences in such a general setting (see~\cite{Graham-Grimmett}).
\end{enumerate}

\vspace{2cm}

\begin{thebibliography}{99}

\bibitem{adler} Adler R. J.,  An Introduction to Continuity, Extrema, and Related Topics for General Gaussian Processes, {\it IMS Lecture Notes-Monograph Series}, \textbf{Vol 12} (1990).

\bibitem{ABGM} D. Ahlberg, E.I. Broman, S. Griffith, and R.
Morris, Noise Sensitivity in Continuum Percolation, submitted,
2011. Available on-line at: www.math.chalmers.se/~broman.

\bibitem{Beckner} W. Beckner, Inequalities in Fourier
Analysis, {\it Annals of Math.} \textbf{102} (1975), pp.~159--182.

\bibitem{BKS} I. Benjamini, G. Kalai, and O. Schramm, Noise
Sensitivity of Boolean Functions And Applications to Percolation,
{\it Publ. I.H.E.S.} \textbf{90} (1999), pp.~5--43.



\bibitem{Bonami} A. Bonami, Etude des Coefficients Fourier
des Fonctiones de $L^p (G)$, {\it Ann. Inst. Fourier} \textbf{20}
(1970), pp.~335--402.


\bibitem{Borell} C. Borell, Positivity Improving Operators and
Hypercontractivity, {\it Math. Zeitschrift,} \textbf{180(2)} (1982),
pp. 225-�234.

\bibitem{Borell2} C. Borell, Geometric Bounds on the
Ornstein-Uhlenbeck Velocity Process, {\it Probab. Th. Rel. Fields}
\textbf{70(1)} (1985), pp.~1--13.

\bibitem{bourgain92} J. Bourgain, J. Kahn, G. Kalai, Y. Katznelson, and
N. Linial, The Influence of Variables in Product Spaces, {\it
Israel J. Math.} \textbf{77} (1992), pp.~55--64.

\bibitem{Sourav} S. Chatterjee, Chaos, Concentration, and Multiple
valleys (2008). Available online at:
http://arxiv.org/abs/0810.4221.

\bibitem{CEL} D. Cordero-Erausquin, M. Ledoux, Hypercontractive Measures,
Talagrand's Inequality, and Influences, preprint, 2011. Available
online at: www.newton.ac.uk/preprints/NI11014.pdf.


\bibitem{Durrett} R. Durrett, Probability: Theory and Examples,
4th Edition, Cambridge University Press, 2010.

\bibitem{FKG} C.M. Fortuin, P.W. Kasteleyn, and J. Ginibre, Correlation
Inequalities on Some Partially Ordered Sets, {\it Comm. Math. Phys.}
\textbf{22} (1971), pp. 89-�103.

\bibitem{Frankl} P. Frankl, The Shifting Technique in Extremal Set
Theory, in {\it Surveys in Combinatorics} (C.W. Whitehead ed.),
Cambridge University Press, Cambridge, 1987, pp.~81--110.




\bibitem{Graham-Grimmett} G.R. Grimmett and B. Graham, Influence
and Sharp-Threshold Theorems for Monotonic Measures, {\it Ann. of
Probability} \textbf{34} (2006), pp.~1726--1745.



\bibitem{Harris} T.E. Harris, A Lower Bound for the Critical
Probability in a Certain Percolation Process, {\it Proc. Cambridge
Phil. Soc.} \textbf{56} (1960), pp.~13--20.


\bibitem{hatami09} H. Hatami, Decision Trees and Influence
of Variables over Product Probability Spaces, {\it Combin.,
Probab., Comp.} \textbf{18} (2009), pp.~357--369.

\bibitem{KKL} J. Kahn, G. Kalai, and N. Linial, The Influence of
Variables on Boolean Functions, Proc. 29-th Ann. Symp. on
Foundations of Comp. Sci., pp.~68--80, Computer Society Press,
1988.

\bibitem{kalai06} G. Kalai and M. Safra, Threshold Phenomena and Influence,
in: Computational Complexity and Statistical Physics,
A.G. Percus, G. Istrate and C. Moore, eds. (Oxford University Press,
New York, 2006), pp.~25--60.



\bibitem{keller09} N. Keller, Influences of Variables on Boolean
Functions, Ph. D. Thesis, Hebrew University of Jerusalem, 2009.

\bibitem{keller11} N. Keller, On the Influences of Variables on Boolean Functions
in Product Spaces, {\it Combin., Probab. Comp.} \textbf{20(1)} (2011), pp.~83--102.

\bibitem{Keller-Reduction} N. Keller, A Simple Reduction from the Biased
Measure on the Discrete Cube to the Uniform Measure, {\it European J. of Comb.},
to appear. Available online at
http://arxiv.org/abs/1001.1167.

\bibitem{Keller-Kindler} N. Keller and G. Kindler, A Quantitative
Relation Between Influences and Noise Sensitivity, {\it Combinatorica},
to appear. Available online at
http://arxiv.org/abs/1003.1839.

\bibitem{Geom-Influences} N. Keller, E. Mossel, and A. Sen, Geometric
Influences, {\it Ann. of Probability}, \textbf{40(3)} (2012),
pp.~1135--1166.

\bibitem{Guy-Ryan} G. Kindler and R. O'Donnell, Gaussian Noise
Sensitivity and Fourier Tails, to appear in {\it CCC'2012}.
Available online at:
http://www.cs.cmu.edu/~odonnell/papers/gaussian-noise-sensitivity.pdf

\bibitem{Kleitman} D.J. Kleitman, Families of Non-Disjoint
Subsets, {\it J. Combin. Theory} \textbf{1} (1966), pp.~153--155.

\bibitem{Ledoux2} M. Ledoux, The Geometry of Markov Diffusion
Generators, Annales -- Faculte des Sciences Toulouse
Mathematiques, \textbf{9(2)} (2000), pp.~305--366.



\bibitem{Maj-Stablest} E. Mossel, R. O'Donnell and K.
Oleszkiewicz, Noise Stability of Functions With Low Influences:
Invariance and Optimality, {\it Annals of Math.}, \textbf{171(1)}
(2010), pp.~295--341.

\bibitem{Mossel-Inverse} E. Mossel, R. O'Donnell, O. Regev, J.E.
Steif, and B. Sudakov, Non-Interactive Correlation Distillation,
Inhomogeneous Markov Chains, and the Reverse Bonami-Beckner
Inequality, {\it Israel J. Math.} \textbf{154} (2006),
pp.~299--336.

\bibitem{Ryan-Survey} R. O'Donnell, Some Topics in Analysis of
Boolean Functions, in {\it Proceedings of the 40th Annual ACM
Sympsium on the Theory of Computing} (2008), pp.~569--578,.






\bibitem{Talagrand1} M. Talagrand, On Russo's Approximate Zero-One
Law, {\it Ann. of Probab.} \textbf{22} (1994), pp.~1576--1587.

\bibitem{Talagrand2} M. Talagrand, How Much are Increasing
Sets Positively Correlated?, {\it Combinatorica} \textbf{16}
(1996), no. 2, pp.~243--258.

\bibitem{Tsirelson} Tsirelson, B. S., Ibragimov, I. A., and Sudakov, V. N. . Norms of Gaussian sample functions, {\it Proceedings of the Third Japan-USSR Symposium on Probability
Theory (Tashkent, 1975)} \textbf{550} (1976),  pp.~20--41.


\bibitem{Wolff} P. Wolff, Hypercontractivity of Simple Random
Variables, {\it Studia Mathematica}, \textbf{180(3)} (2007),
pp.~219--236.

\end{thebibliography}
\end{document}